\documentclass[final]{siamltex}

\usepackage{framed}
\usepackage{xspace}
\usepackage{amsmath}		
\usepackage{amsfonts}		
\usepackage{amssymb}
\usepackage{latexsym}		
\usepackage{bbm}	               
\usepackage[latin1]{inputenc}
\usepackage{amssymb,theorem,setspace}
\usepackage{multirow}
\usepackage{float}
\usepackage{afterpage}
\usepackage{url}
\usepackage{cite}
\usepackage{rotating}
\usepackage{color}
\usepackage[all]{xy}

\usepackage{algorithm}          
\floatname{algorithm}{Alg.}     
\usepackage{algorithmic}        

\theoremstyle{plain}
\newtheorem{theo}{Theorem}[section]
\newtheorem{lem}[theorem]{Lemma}

\newtheorem{defi}[theorem]{Definition}
\newtheorem{prop}[theorem]{Proposition}
\newtheorem{example}[theorem]{Example}
\newtheorem{remark}[theorem]{Remark}

\newtheorem{assump}{Assumption}

\DeclareMathOperator{\Ad}{Ad}
\DeclareMathOperator{\ad}{ad}

\DeclareMathOperator{\Id}{Id}

\newcommand{\R}{\mathbb{R}}

\newcommand{\mc}[1]{\mathcal{#1}}

\newcommand{\ud}{\,\mathrm{d}}
\providecommand{\abs}[1]{\lvert#1\rvert}

\newcommand{\ms}[1]{\mathbb{#1}}
\newcommand{\g}{\mathfrak{g}}

\def\o{\circ}

\def\la{\lambda}

\def\ph{\varphi} 
 
\def\ps{\psi}

\def\Ph{\Phi} 
\def\Ps{\Psi} 
\def\Om{\Omega} 
\def\i{^{-1}} 
\def\x{\times} 
\def\g{{\mathfrak g}} 
\def\p{\partial}
\def\ad{\operatorname{ad}} 
\def\Ad{\operatorname{Ad}}

\let\on=\operatorname

\graphicspath{{.}{./IMAGES/}}

\title{Mixture of Kernels and Iterated semidirect Product of Diffeomorphisms Groups.\thanks{ERC grant \emph{Five Challenges in Computational Anatomy}}}

\author{
Martins Bruveris\thanks{Department of Mathematics, Imperial College, London SW7 2AZ, UK. \tt{m.bruveris08@ic.ac.uk}} \and 
Laurent Risser\thanks{Institute of Biomedical Engineering, University of Oxford, UK and Institut de Math\'ematiques de Toulouse, UMR CNRS 5219, FR. \tt{lrisser@math.univ-toulouse.fr}} \and 
Fran\c cois-Xavier Vialard\thanks{Ceremade, UMR CNRS 7534, Universit\'e Paris-Dauphine, FR. \tt{vialard@ceremade.dauphine.fr}}}

\begin{document}

\maketitle

\begin{abstract}
In the framework of \emph{large deformation diffeomorphic metric mapping} (LDDMM), we develop a multi-scale theory for the diffeomorphism group based on previous works \cite{Bruveris_MK,Risser_MK,SommerKernelBundle1,SommerKernelBundle2}.
The purpose of the paper is (1) to develop in details a variational approach for multi-scale analysis of diffeomorphisms, (2) to generalise to several scales the semidirect product representation 
and (3) to illustrate the resulting diffeomorphic decomposition on synthetic and real images. We also show that the approaches presented in \cite{SommerKernelBundle1,SommerKernelBundle2} and the mixture of kernels of [RVW+11] are equivalent.
\end{abstract}


\pagestyle{myheadings}
\thispagestyle{plain}

\section{Introduction}
In this paper, we develop a multi-scale theory for groups of diffeomorphisms in the context of image registration. 
Very little work has been done in this direction, but we can mention \cite{Kaplan_themorphlet} which addresses multi-scale on diffeomorphisms with the same goals as wavelets. Our motivations differ from such approaches. In this introduction, we first introduce the context of our work and then present our goals.
The setting of \emph{large deformation diffeomorphic metric mapping} (LDDMM) has been introduced in seminal papers \cite{Trouve1998,DuGrMi1998} and this approach has been applied in the field of computational anatomy \cite{GrMi1998}. The initial problem deals with the diffeomorphic registration of two given biomedical images or shapes. An important aspect of this model is the use of reproducing kernel Hilbert spaces (RKHS) of vector valued functions to define the Lie algebra of the group of diffeomorphisms. 
We now present the model developed in \cite{Begetal2005}.
\begin{defi}
Let $\Omega$ be a domain in $\R^d$. An admissible reproducing kernel Hilbert space $H$ of vector fields is a Hilbert space of $C^1(\Omega)$ vector fields such that there exists a positive constant $M$, s.t. for every $v \in H$ the following inequality holds
\begin{equation}
\| v\|_{1,\infty} \leq M \| v\|_{H}\,.
\end{equation}
\end{defi}

\begin{remark}
Note that our assumption is more demanding than the one defining reproducing kernel Hilbert spaces. Indeed, a reproducing kernel Hilbert space (of vector fields) is a Hilbert space $H$ of functions from $\Omega$ to $\R^d$ such that the pointwise evaluation maps denoted by $\delta_x : f \in H \mapsto f(x) \in \R^d$ are continuous. Denoting $K: H^* \mapsto H$ the Riesz isomorphism between $H^*$ (the dual of $H$) and $H$, the reproducing  kernel associated with the space $H$ is defined by $\mathsf{k}(x,y) = (\delta_x, K \delta_y) \in L(\R^d,\R^d)$, where the bracket $(\cdot,\cdot)$ denotes the dual pairing. In other words, for any points $x,y \in \Omega$, the kernel $\mathsf{k}(x,y)$ is a linear map from $\R^d$ to itself. The kernel completely specifies the reproducing kernel Hilbert space $H$: we refer the reader to \cite{saitohbook} for more informations on RKHS. 
\end{remark}
\begin{example}
In practice, a Gaussian kernel $\mathsf{k}(x,y) = e^{-\|x-y\|_{\R^d}^2/\sigma^2}\on{Id}_{\R^d}$ is often used and we will call $\sigma$ the width of the Gaussian kernel. The norm of a vector field $v$ in the corresponding RKHS $H$ can be computed via a Fourier transform by
\begin{equation}
\|v\|_{H}^2 = \|\hat{\mathsf{k}}^{-1/2}\hat{v}\|_{L^2}^2\,,
\end{equation}
where $\hat{f}$ denotes the Fourier transform of $f$.
\end{example}
The diffeomorphism group $G$ associated with the RKHS $H$ is given by 
the flows of all time-dependent vector fields in $L^2([0,1],H)$. In order to have a well-defined flow, we assume that the vector fields $v \in H$ vanish on the boundary of $\Omega$. Hence, this boundary will be fixed by the flow. We assume implicitly that hypothesis in the rest of the paper. More precisely, we define

\begin{equation}G:=\left\{ \varphi(1) \, | \, v \in L^2([0,1],H)\right\},
\end{equation}
where $\varphi(1)$ is the flow at time $1$ of the vector field $v$, {\it i.e.}
\begin{equation}\label{eq:integrationVF}
\begin{cases}
\partial_t \varphi(t) = v(t) \circ \varphi(t) \\
\varphi(0) = \on{Id}\,.
\end{cases}
\end{equation}

The main idea of the LDDMM approach is to deform the objects of interest via a deformation of the whole ambient space. Therefore, an action of the group $G$ on the object space $Q$ is introduced and denoted by $\varphi . I$, where $\varphi$ is an element of the group and $I$ is an object. The set of objects of interest can be of various types, such as groups of points, measures, currents or images. We also assume that there exists a distance on these spaces: for example the usual distance on a normed vector space, e.g. for images, one would use the $L^2$ norm: $d^2(I_0, I_1) = \int_\Om |I_0(x)-I_1(x)|^2 dx$. One motivation underlying diffeomorphic matching via large deformations is to quantify the geometric variability of biological shapes and their changes. The first common step consists in solving the diffeomorphic matching problem, which reduces to the minimisation of the functional
\begin{equation} \label{MatchingFunctional}
\mathcal{F}(v) = \frac 12  \int_0^1\! \|v(t)\|_{H}^2 \, \ud t + d^2(\varphi(1).I_0,I_{\on{target}}) \,,
\end{equation}
where $v \in L^2([0,1],H)$ and $I_0,I_{\on{target}} \in Q$ and the distance function $d$ enforces the matching accuracy. This minimisation problem enables to match images via geodesics on the group $G$, if $G$ is endowed with the right-invariant metric obtained by translating the inner product $\langle.,.\rangle_H$ on the Lie algebra $H$ to the other tangent spaces. More importantly, by its action on the space of images, the right-invariant metric on the group induces a metric on the orbits of the image space and the final deformation is completely encoded in the so-called initial momentum \cite{vmty04}. This initial momentum has the same dimension as the image. Since it is an element of a linear space, it can be used to perform statistics on it \cite{LNCS63630529}.

\begin{figure}[htb]
\begin{center}
\begin{tabular}{cc}
\begin{tabular}{cc}
$\null$\hspace{-0.35cm}
\includegraphics[width= 2.5cm]{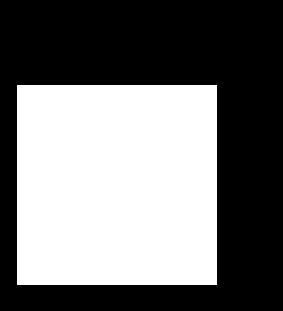}
&
$\null$\hspace{-0.5cm}
\includegraphics[width= 2.5cm]{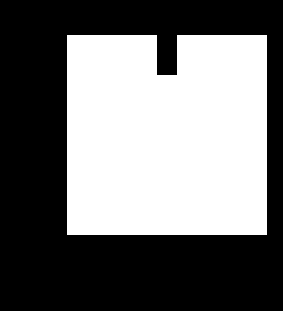}
\\
Source Image $I_S$&
Target Image $I_T$
\end{tabular}
&
$\null$\hspace{-0.6cm}
\begin{tabular}{ccc}
$\null$\hspace{-0.35cm} \includegraphics[width= 2.1cm]{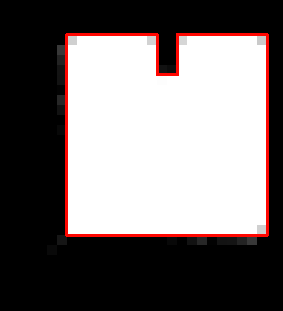}&
$\null$\hspace{-0.45cm} \includegraphics[width= 2.1cm]{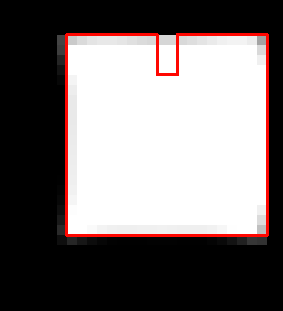}&
$\null$\hspace{-0.45cm} \includegraphics[width= 2.1cm]{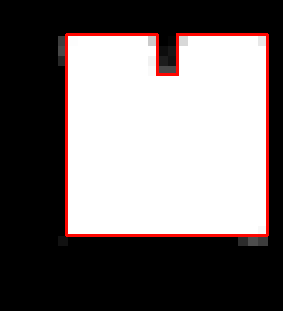}\\ 
$\null$\hspace{-0.35cm} \includegraphics[width= 2.1cm]{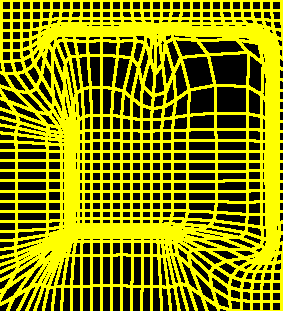}&
$\null$\hspace{-0.45cm} \includegraphics[width= 2.1cm]{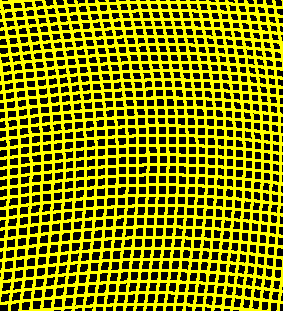}&
$\null$\hspace{-0.45cm} \includegraphics[width= 2.1cm]{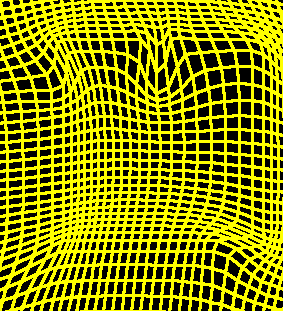}\\ 
$\null$\hspace{-0.35cm}  K$_{1}$   &
$\null$\hspace{-0.45cm}  K$_{10}$ &
$\null$\hspace{-0.45cm}  MK5 \\ 
\end{tabular}
\\
{\bf (a)}&{\bf (b)}
\end{tabular}
\caption{
Influence of the smoothing kernel when registering two images containing feature differences at several scales simultaneously in the LDDMM framework (from \cite{Risser_MK}).
{\bf (a)} Region of interest of the  source and target images $I_S$ and $I_T$ containing the registered shapes.
{\bf (b)} Registration of the images $I_S$ and $I_T$ using different kernels: ($\mathsf{k}_{1}$ and $\mathsf{k}_{10}$) 
Gaussian kernels of width $\sigma=1$ and $\sigma=10$ pixels; (MK5) sum of 5  Gaussian kernels linearly sampled between $10$ and $1$ pixels. 
Diffeomorphic transformations of $I_S$ at $t=1$ (final deformation) are shown on the top and corresponding homogeneous grids (step = 1 pixel) are on the bottom.
}
\label{IllusBase}
\end{center}
\end{figure}

In this article, we  particularly focus on the choice of the Lie algebra, the RKHS of vector fields $H$. In theory, as soon as the RKHS of vector fields contains enough functions so that the Stone-Weierstrass theorem can be applied, the generated group of diffeomorphisms will be, loosely speaking, dense in $\on{Diff}_0(\Omega)$, the group of diffeomorphisms of $\Omega$ that leave the boundary $\partial \Omega$ fixed. As a consequence, the matching will be as accurate as desired, according to the weight of the matching term, for a large class of underlying deformations. In practice, the choice of the metric on the Lie algebra is however critical. The main reason is that solving \eqref{MatchingFunctional} for different choices of RKHS can produce equally good deformations according to the similarity measure but the resulting deformations may present a wide variety of forms. This feature is common to inverse problems, where the prior or regularising term (in our case, the RKHS) is not known. In this case many priors can be chosen in order to turn the problem into a well-posed minimisation problem. Another practical issue is the numerical convergence of the algorithm. Let us explain this in the case of Gaussian kernels, which is the standard choice for solving \eqref{MatchingFunctional}. In that case, the RKHS associated with the Gaussian kernels of decreasing width form an increasing sequence of spaces. The choice of this width $\sigma$ describes a trade-off between the smoothness of the deformation and the  matching accuracy. For example, a large width produces very smooth deformations with a poor matching  accuracy of the structures having a size smaller than $\sigma$. Indeed, the cost of fine deformations is high and it cannot be achieved in practice. 
On the contrary,  a small width results in a good matching accuracy but the resulting deformations will have Jacobians with a large determinant, which is also undesirable. It is therefore a natural step to introduce a mixture of Gaussian kernels with different widths. 
By using a mixed kernel, constructed as the weighted sum of several Gaussian kernels the estimated deformations can be smooth and provide a good matching quality.
This is one of the main results of \cite{Risser_MK} that we illustrate in Fig.~\ref{IllusBase}.
Naturally, using mixed kernels introduces more parameters in the algorithm, which need to be tuned. Practical insights about how to parametrise the scales and weights of multiple kernels are given in  \cite{Risser_MK}. 
The idea of using a mixture of kernels for matching is also directly connected to \cite{Bruveris_MK}, where it is proven that there is an equivalence between the matching with a sum of two kernels and the matching via a semidirect product of two groups. 
The work on the metric underlying the LDDMM methodology \cite{Begetal2005} has also been followed up by \cite{SommerKernelBundle1,SommerKernelBundle2}, where the authors introduce the notion of a bundle of kernels and argue that this general framework can be used to add a multi-scale approach to LDDMM. In passing, we prove that their approach reduces to the mixture of kernels. We give a self-contained and simple proof of this result based on Lagrange multipliers.\\

We emphasise that our work deals with the multi-scale properties of the smoothing kernel $\mathsf{k}$. This is different from insights about standard multi-resolution algorithms. In these algorithms, the registered images are subsampled to a higher or lesser degree, as opposed to the smoothing kernels
which remain the same whatever the image resolution. The regularisation of the deformation therefore depends on the amount of sub-sampling. A mutli-resolution algorithm is also usually much faster than a single-resolution one.
Multi-resolution strategies can be adopted in the LDDMM framework,  but the smoothing kernel $\mathsf{k}$ should be discretised with the same level of coarseness as the images. The metric  on $H$  is indeed related to the compared structures and not to the  image resolution.  This justifies the definition of multi-scale kernels to compare real images having feature differences at several scales simultaneously.\\

We  also address a question of crucial interest which is the description of the influence of each scale on the final deformation. This question is non-trivial due to the fact that the group of diffeomorphisms is not a linear space, so that standard multi-scale methods do not apply - they would ignore the group structure of the space. In particular, since our multi-scale approach is developed on the Lie algebra $H$, we must develop a framework to decompose the final diffeomorphism into separate scales. This is of great practical importance: 
For instance, we may be interested in the \lq\lq{}high-frequency\rq\rq{} deformations since the \lq\lq{}low-frequency\rq\rq{} deformations may be biased due to an initial rigid registration. In addition, there is no completely established justification of what is an optimal rigid alignment between two brain images although their comparison at a fine scale (about 1 millimetre) is of high interest in neuroimaging.  
 To this end, we develop an extension of the semidirect product of groups of \cite{Bruveris_MK}, first to more than two discrete scales and then to a continuum of scales. This approach introduces a decomposition of the final deformation into several diffeomorphisms at each scale of interest. 
Using this decomposition, we may extract more meaningful statistical information as shown in the simulation section of our paper.\\

This article is divided into three parts: the first part focuses on a finite number of scales while the second treats the case of a continuum of scales. The last part of the paper is devoted to numerical simulations, where we show in particular the decomposition on the given scales of the optimised diffeomorphism.

\section{A finite number of scales} \label{FiniteNumberScales}
\subsection{The finite mixture of kernels}
For the sake of simplicity, we first treat the case of a finite set of admissible Hilbert spaces $H_i$ with kernels $\mathsf{k}_i$ and Riesz isomorphisms $K_i$ between $H_i^*$ and $H_i$ for $i =  1, \ldots, n$.
Denoting $H = H_1 + \ldots + H_n$, the space of all functions of the form $v_1+\ldots+v_n$ with $v_i \in H_i$, the norm proposed in \cite{SommerKernelBundle1} as well as in \cite{Bruveris_MK} is defined by
\begin{equation} \label{InfNorm}
\|v\|_{H}^2 = \inf \left\{ \sum_{i=1}^n \|v_i\|_{H_i}^2 \, \Big | \, \sum_{i=1}^n v_i = v \right \}\,.
\end{equation}
 The minimum is achieved for a unique $n$-tuple of vector fields and the space $H$, endowed with the norm defined by \eqref{InfNorm}, is complete. 
The following lemma is the main argument to prove the equivalence between the approaches of \cite{SommerKernelBundle1} and \cite{Risser_MK}. This lemma is an old result that can be found in \cite{Aronszajn1950}.
However, for the sake of completeness, we present a simple proof based on the Lagrange multiplier rule. Moreover, if one wants to skip the technical details of the proof, the formal application of this variational calculus theorem immediately gives the result. We outline the formal proof of the lemma: If one has to minimise the sum defined in Formula \eqref{InfNorm} then one can introduce a Lagrange multiplier $p$ and obtain a stationary point of the Lagrangian
\begin{equation}
\ell_v(v_1, \ldots,v_n,p) = \frac 12 \sum_{i=1}^n \|v_i\|_{H_i}^2 + \left(p,v-\sum_{i=1}^n v_i\right)\,,
\end{equation}

where the notation $(\cdot,\cdot)$ stands for the dual pairing.
Therefore, one has $v_i = K_i p$ and
\begin{equation}
v = \sum_{i=1}^n K_ip \, .
\end{equation}
This formally shows that optimising at several scales simultaneously reduces to a mixture of kernels, since the kernel is then given by $\mathsf{k} = \sum_{i=1}^n \mathsf{k}_i$.
\begin{lem}\label{ReductionLemma}
The formula \eqref{InfNorm} induces a scalar product on $H$ which makes $H$ a RKHS, and its associated kernel is $\mathsf{k} := \sum_{i=1}^n \mathsf{k}_i$, where $\mathsf{k}_i$ denotes the kernel of the space $H_i$.
\end{lem}

\begin{proof}
Let $x \in \Omega$, $\alpha \in \R^d$ and $\delta_x^\alpha$ be the pointwise evaluation defined by $\delta_x^\alpha(v) \doteq \langle v(x),\alpha \rangle_{\R^d}$. By hypothesis on each $H_i$, $\delta_x^\alpha$ is a linear form which implies that $$\mathrm{ev}_{(x,\alpha)} : (v_i)_{i=1 \ldots n} \in \bigoplus_{i=1}^n H_i \mapsto \sum_{i=1}^n \delta_x^\alpha(v_i) \in \ms{R}$$ is also a linear form on $\bigoplus_{i=1}^n H_i$. Therefore we see that the intersection of the kernels $S:=\bigcap_{(x,\alpha) \in \Omega \times \R^d} \mathrm{ev}_{(x,\alpha)}^{-1}(\{0\})$ is a closed subspace of $\bigoplus_{i=1}^n H_i$.
Let $\pi$ be the orthogonal projection on $S^{\perp}$. For any $v \in H$ that can be written as $v = \sum_{i=1}^n u_i$ with $(u_i)_{i=1 \ldots n} \in \bigoplus_{i=1}^n H_i$, there exists a unique $(v_i)_{i=1 \ldots n} \in \bigoplus_{i=1}^n H_i$ minimising the functional $N\left((v_i)_{i=1 \ldots n} \right) = \frac{1}{2} \sum_{i=1}^n \|v_i\|_{H_i}^2$ and satisfying $\sum_{i=1}^n v_i = v$.
This unique element is given by $\pi((u_i)_{i=1 \ldots n})$ as a consequence of the  projection theorem for Hilbert spaces \cite{bre83}. Therefore $H$ is isometric to $S^\perp$ and hence $H$ is a Hilbert space.

We now introduce the Lagrangian 
\begin{equation}\label{AugmentedLagrangian}
\ell_v(v_1, \ldots,v_n,p)  =\frac{1}{2} \sum_{i=1}^n \|v_i\|_{H_i}^2 + \left(p,v-\sum_{i=1}^n v_i \right)_{H^*,H}\,.
\end{equation}
defined on $\bigoplus_{i=1}^n H_i \times H^*$. Remark that the norm $\|\cdot \|_{H}$ makes the injection $j_i:H_i \hookrightarrow H$ continuous and as a consequence $j_i^*:H^* \rightarrow H_i^*$ is defined by duality. Therefore the pairing $(p,v_i)$ in Formula \eqref{AugmentedLagrangian} is well-defined. In addition, $\ell_{v}$ is also Fr\'echet differentiable.
One can easily check that, for a given $v \in H$, a stationary point of the Lagrangian
is $(\pi((v_i)_{i=1 \ldots n}),p)$ where $Kp=v$ and $K$ is the Riesz isomorphism between $H^*$ and $H$.
Then, at this stationary point, we have 
\begin{equation}
\begin{cases}
v_i = K_i(p)  \text{ for } i = 1 \ldots n \\
v = \sum_{i=1}^n  v_i = \sum_{i=1}^n K_i(p) \,.
\end{cases}
\end{equation}
Note that in the previous formula, we could have written the heavier notation $v_i = K_i(j_i^*p)$ to be more precise.
This implies that the Riesz isomorphism between $H^*$ and $H$ is given by the map $p \in H^* \mapsto \sum_{i=1}^n K_i(p) \in H$. 
Moreover, we have $ \bigcap_{i=1}^n H_i^* \subset H^*$: For $p \in \bigcap_{i=1}^n H_i^*$ we have,
\begin{multline*}
|(p,v)| \leq \sum_{i=1}^n |(p,v_i)| \leq \sum_{i=1}^n \|p\|_{H_i^*} \|v_i\|_{H_i}  \stackrel{C.S.}{\leq} \sqrt{\sum_{i=1}^n  \|p\|_{H_i^*}^2} \sqrt{\sum_{i=1}^n \|v_i\|_{H_i}^2}\,,
\end{multline*}
which is true for any decomposition of $v$ so that
\begin{equation}
|(p,v)| \leq \sqrt{\sum_{i=1}^n \|p\|_{H_i^*}^2} \|v\|_H\,.
\end{equation}
Since $\delta_x^\alpha  \in \bigcap_{i=1}^n H_i^* \subset H^*$, $H$ is a RKHS and its kernel is $\mathsf{k}=\sum_{i=1}^n \mathsf{k}_i$.
\end{proof}

We now define the isometric injection of $H$ in $\bigoplus_{i=1}^n H_i$ which is simply the inverse of the projection $\pi$ introduced in the proof above.
\begin{defi}
We denote by $\pi^{-1} : H \mapsto \bigoplus_{i=1}^n H_i$ the map defined by $\pi^{-1}(v) = \left(K_i (K^{-1} (v) )\right)_{i=1 \ldots n}$.
\end{defi}

The non-linear version of this multi-scale approach to diffeomorphic matching problems is the minimisation of
\begin{equation} \label{MultiScaleMatching}
\mathcal{E}(v) =\frac 12  \int_0^1\! \sum_{i=1}^n \|v_i(t)\|_{H_i}^2 \, \ud t + d^2(\varphi(1).I_0,I_{\on{target}}) \,,
\end{equation}
defined on $\bigoplus_{i=1}^n H_i $. Recall that $\varphi(t)$ is the flow generated by $v(t) := \sum_{i=1}^n v_i(t)$. 
The direct consequence of Lemma \ref{ReductionLemma} is the following proposition:
\begin{prop}
\label{EquivalenceMatching}
The minimisation of $\mathcal{E}$ reduces to the minimisation of 
\begin{equation} \label{ReducedMatching}
\mathcal{F}(v) =\frac 12  \int_0^1\! \|v(t)\|_{H}^2 \, \ud t + d^2(\varphi(1).I_0,I_{\on{target}}) \,.
\end{equation}
\end{prop}

\begin{proof}
Obviously, minimising $\mathcal{F}$ is the same as minimising $\mathcal{E}$ restricted to $\pi^{-1}(H)$.
Note first that for any $n$-tuple $(v_i)_{i=1 \ldots n} \in L^2([0,1],\bigoplus_{i=1}^n H_i)$ we have $\pi((v_i)_{i=1 \ldots n}) \in  L^2\left([0,1],\bigoplus_{i=1}^n H_i\right)$. Denoting  $v = \sum_{i=1}^n v_i$ we have $\| \pi^{-1}(v) \|_{L^2} \leq \| (v_i)_{i=1 \ldots n} \|_{L^2}$
with equality if and only if $\pi^{-1}(v) = (v_i)_{i=1 \ldots n}$.
Therefore, it follows that if $(v_i)_{i=1 \ldots n} \in L^2([0,1],\bigoplus_{i=1}^n H_i)$ is a minimiser of $\mathcal{E}$ then
\begin{equation}
\mathcal{E}(\pi(v)) \leq \mathcal{E}((v_i))\,,
\end{equation}
which implies $\pi(v) = (v_i)_{i=1 \ldots n}$ and the result.
\end{proof}

\begin{remark}
To a minimising path $v(t) \in L^2([0,1],H)$ corresponds a minimising path in $\bigoplus_{i=1}^n H_i $ via the map $\pi^{-1}$. In other words, the optimal path can be decomposed on the different scales using each kernel $\mathsf{k}_i$ present in the reproducing kernel of $H$, $\mathsf{k} = \sum_{i=1}^n \mathsf{k}_i$.
\end{remark}

\subsection{Iterated semidirect product}
\label{sec_sdp}

Until now, the scales have been introduced only on the Lie algebra of the diffeomorphism group and a remaining question is how to decompose the flow of diffeomorphisms according to these scales. An answer in the case of two scales is given in \cite{Bruveris_MK}, where the flow is decomposed with the help of a semidirect product of groups and the whole transformation is encoded in a large-scale deformation and a small-scale one. The underlying idea is to represent the flow of diffeomorphisms $\ph(t)$ by a pair $(\ph_1(t),\ph_2(t))$ where $\ph_1(t)$ is given by the flow of the vector field $v_1(t)$ and $\ph_2(t) := \ph(t) \circ \left(\ph_1(t)\right)^{-1}$. Note in particular that $\ph_2(t)$ is not the flow of $v_2(t)$. More precisely, we have
\begin{equation}
\begin{cases}
\partial_t \ph_1(t) = v_1(t) \circ \ph_1(t) \\
\partial_t \ph_2(t) = (v_1(t)+v_2(t)) \circ  \ph_2(t) - \Ad_{\ph_2(t)}v_1(t) \circ \ph_2(t) \,.
\end{cases}
\end{equation}

The last equation can be derived as follows,
\begin{align}
\partial_t \ph_2(t) &= \p_t \left(\ph(t) \o \ph_1(t)\i\right) \\
&= \p_t \ph(t) \o \ph_1(t)\i - D\ph(t).D\ph_1(t)\i.\left(\p_t \ph_1(t) \o \ph_1(t)\i \right) \\
&= (v_1(t) + v_2(t)) \o \ph_2(t) - D\ph_2(t).v_1(t) \\
&= (v_1(t) + v_2(t)) \o \ph_2(t) - \Ad_{\ph_2(t)}v_1(t) \circ \ph_2(t) \,.
\end{align}
Here $\Ad_\ph v$ denotes the adjoint action of the group of diffeomorphisms on the Lie algebra of vector fields and is given by
\begin{equation}
\Ad_\ph v(x) = (D\ph.v) \o \ph\i (x) = D_{\ph\i(x)}\ph.v(\ph\i(x))\,.
\end{equation} 
We assume that $\ph_1(t)$ contains the small-scale information and $\ph_2(t)$ the large-scale deformations.  Interestingly, as shown in \cite{Bruveris_MK}, this decomposition of the diffeomorphism flow corresponds to a semidirect product of groups. This framework can be generalised to a finite number of scales as follows.

Given $n$ scales, we want to represent $\varphi(t)$ by an $n$-tuple $(\ph_1(t),\ldots,\ph_n(t))$, with $\ph_1(t)$ corresponding to the finest scale and $\ph_n(t)$ to the coarsest scale. The geometric construction underlying the decomposition into multiple scales is the semidirect product, introduced in the following lemma. We want to consider $n$-tuples $(\ph_1,\ldots,\ph_n)$ as representing the diffeomorphism $\ph_1\o\dots\o\ph_n$. Given two $n$-tuples $(\ph_1,\ldots,\ph_n)$ and $(\ps_1,\ldots,\ps_n)$ the semidirect product multiplication is defined in such a way, that their product represents the concatenated diffeomorphism $\ph_1\o\dots\o\ph_n\o\ps_1\o\dots\o\ps_n$.

\begin{lem}
\label{lem_sdp1}
Let $G_1 \supseteq G_{2} \supseteq \ldots \supseteq G_n$ be a chain of Lie groups. One can define the $n$-fold semidirect product multiplication on the set $G_1\x\dots \x G_n$ via
\begin{equation}
 (g_1, \ldots, g_{n}) \cdot (h_1, \ldots, h_{n}) = (g_1 \on{c}_{g_2\cdots g_{n}}h_1, g_2 \on{c}_{g_3\cdots g_{n}} h_2, \ldots, g_{n}h_{n})\,,
\end{equation}
with $ \on{c}_g h = ghg\i$ denoting conjugation. Then given the right-trivialised tangent vector $v(t) = (v_1(t),\ldots,v_n(t))$ of the curve $g(t)=(g_1(t),\ldots,g_n(t))$, the curve can be reconstructed via the ODE 
\begin{equation}
\label{sdp1_rec}
 \p_t g_k(t) = \left( v_k(t) + (\Id-\Ad_{g_k(t)}) \sum_{i=k+1}^{n} v_i(t) \right) g_k(t)\,,
\end{equation}
 if $k < n$
and $\p_t g_n(t) = v_n(t)g_n(t)$. Here $v_kg_k \in T_{g_k} G_k$ denotes the action of the group on its tangent space obtained by differentiating the left-multiplication.  We shall denote this semidirect product by
\begin{equation} 
G_1 \rtimes \dots \rtimes G_n
\end{equation}
to emphasise that each sub-product $G_1 \rtimes \dots \rtimes G_k$ is a normal subgroup of the whole product.
\end{lem}

\begin{proof}
Verifying the axioms for the group multiplication is a straight forward, if slightly longer calculation. The inverse is given by
\[ (g_1,\ldots,g_n)\i = (\on{c}_{(g_2\cdots g_n)\i} g_1\i,\ldots,g_n\i)\enspace. \]
The right hand side of equation \eqref{sdp1_rec} can be obtained by differentiating the group multiplication at the identity, i.e. computing $\p_t (h(t)\cdot g)|_{t=0}$ with $g$ fixed, $h(0)=\Id$ and $\p_t h(t)|_{t=0} = v$. Step-by-step the computation is as follows.
\begin{align}
\p_t (h(t) \cdot g)_k|_{t=0} &= \p_t h_k(t) \on{c}_{h_{k+1}(t)\cdots h_{n}(t)}g_k|_{t=0} \\
&= v_k g_k + \sum_{i=k+1}^n \p_t \on{c}_{h_i(t)} g_k|_{t=0} \\
&= v_k g_k + \sum_{i=k+1}^n \p_t h_i(t) g_k h_i(t)\i|_{t=0} \\
&= v_k g_k + \sum_{i=k+1}^n v_i g_k - g_k v_i \\
&= v_k g_k + \sum_{i=k+1}^n \left(\on{Id} - \Ad_{g_k}\right) v_i g_k \\
\end{align}
This concludes the proof.
\end{proof}

In our case the group $G_k$ is the diffeomorphism group $\on{Diff}_{k}(\Om)$ generated by vector fields in the space $H_{k}$ corresponding to the kernel $\mathsf{k}_k$. The subgroup condition $\on{Diff}_{k}(\Om) \supseteq \on{Diff}_{k+1}(\Om)$ is satisfied, if we impose the corresponding condition $H_k \supseteq H_{k+1}$ on the spaces of vector fields, which we will assume from now on.

Starting from an $n$-tuple $v_1(t),\ldots,v_n(t)$ of vector fields, we can reconstruct the diffeomorphisms at each scale via
\begin{equation}
\label{rec_sdp1}
 \p_t \ph_k(t) = \left( v_k(t) + (\Id-\Ad_{\ph_k(t)}) \sum_{i=k+1}^n v_i(t) \right) \o \ph_k(t)\,,
\end{equation}
as in Lemma \ref{lem_sdp1}. We can also sum over all scales to form $v(t) = \sum_{k=1}^n v_k(t)$ and compute the flow $\ph(t)$ of $v(t)$. Then a simple calculation shows that
\begin{equation}
 \ph(t) = \ph_1(t) \o \dots \o \ph_n(t)\,.
\end{equation}
This construction can be summarised by the following commutative diagram
\begin{equation}
\label{comm_findim}
\vcenter{ \xymatrix{
(v_1(t),\ldots,v_n(t))\ar[rrr]^<>(0.5){v(t) = \sum v_k(t)}\ar[d]_{\text{via \eqref{rec_sdp1}}} &&& v(t) \ar[d]^{\p_t \ph(t) = v(t) \o \ph(t)} \\
(\ph_1(t),\ldots,\ph_n(t))\ar[rrr]_<>(0.5){\ph(t)= \ph_1(t) \o \dots \o \ph_n(t)} &&& \ph(t)
}}
\end{equation}

We can now formulate several equivalent versions of LDDMM matching with multiple scales. The most straight-forward way is to do matching with a kernel which is a sum of kernels of different scales. This is the approach considered in \cite{Risser_MK}. 
\begin{defi}[LDDMM with sum-of-kernels]
\label{lddmm_sumofker}
Registering the image $I_0$ to the image $I_{\on{target}}$ is done by finding the minimiser $v(t)$ of
\[  \frac 12 \int_0^1 \| v(t) \|^2_{H} \ud t + d^2( \ph(1).I_0,I_{\on{target}})\enspace,\]
where $\ph(t)$ is the flow of  $v(t)$ and $H$ is the RKHS with kernel $\mathsf{k}=\sum_{i=1}^n \mathsf{k}_i$.
\end{defi}
The corresponding simultaneous multiscale registration problem, where one assigns to each scale a separate vector field, is a special case of the kernel bundle method proposed in \cite{SommerKernelBundle1} and \cite{SommerKernelBundle2}.
\begin{defi}[Simultaneous multiscale registration]
\label{lddmm_simmultscale}
Registering $I_0$ to $I_{\on{target}}$ is done by finding the minimising $n$-tuple $(v_1(t),\ldots,v_n(t))$ of
\[  \frac 12 \sum_{i=1}^n \int_0^1 \| v_i(t) \|^2_{H_i} \ud t + d^2( \ph(1).I_0,I_{\on{target}})\enspace,\]
where $\ph(t)$ is the flow of the vector field $v(t) = \sum_{i=1}^n v_i(t)$.
\end{defi}
The geometric version of the multiscale registration not only uses separate vector fields for each scale, but also decomposes the diffeomorphisms according to scales and can be defined as follows.
\begin{defi}[LDDMM with a semidirect product]
\label{lddmm_sdp1}
Registering $I_0$ to $I_{\on{target}}$ is done by finding the minimising $n$-tuple $(v_1(t),\ldots,v_n(t))$ of
\[  \frac 12 \sum_{i=1}^n \int_0^1 \| v_i(t) \|^2_{H_i} \ud t + d^2( \ph(1).I_0,I_{\on{target}})\enspace,\]
where $\ph(t) = \ph_1(t)\o\dots\o\ph_n(t)$ and $\ph_i(t)$ is defined via \eqref{rec_sdp1}.
\end{defi}
Problem \ref{lddmm_sdp1} can be obtained from the abstract framework in \cite{Bruveris_MK} by considering the action
\begin{equation} \label{sdp1_action} \begin{array}{ccc}
 \left(\on{Diff}_{1}(\Om) \rtimes \dots \rtimes \on{Diff}_{n}(\Om)\right) \times V &\to& V\\
 ((\ph_1,\ldots,\ph_n), I) &\mapsto & I\o\ph_n\i\o\dots\o\ph_1\i
\end{array} \end{equation}
of the semidirect product on the space of images.

\begin{theo}
The matching problems \ref{lddmm_sumofker}, \ref{lddmm_simmultscale} and \ref{lddmm_sdp1} are all equivalent.
\end{theo}

\begin{proof}
The equivalence of problems \ref{lddmm_sumofker} and \ref{lddmm_simmultscale} follows from Proposition \ref{EquivalenceMatching} while the equivalence of problems \ref{lddmm_simmultscale} and \ref{lddmm_sdp1} follows from the diagram \eqref{comm_findim}. For the case $n=2$ the proof can be found in more detail in \cite{Bruveris_MK}.
\end{proof}

This construction can also be generalised to a continuum of scales as introduced in \cite{SommerKernelBundle1} and \cite{SommerKernelBundle2}. We present in the Section \ref{cont_sec} a more general framework to deal with such continuous multi-scale decompositions.

\subsection{The order reversed}
\label{sec_reverse}

The action \eqref{sdp1_action} of the semidirect product from Lemma \ref{lem_sdp1} proceeds by deforming the image with the coarsest scale diffeomorphism first and with the finest scale diffeomorphism last. However, it is also possible to reverse this order and to act with the finest scale diffeomorphisms first. We will see that this approach also corresponds to a semidirect product and is equivalent to the other ordering of scales via a change of coordinates. The reason to expand on this here is that this version is better suited to be generalised to a continuum of scales.

In this section we will assume that the group $G_1$ contains the deformations of the coarsest scale and $G_n$ those of the finest scale. The corresponding semidirect product is described in the following lemma.

\begin{lem}
\label{lem_sdp2}
Let $G_1 \subseteq G_{2} \subseteq \ldots \subseteq G_n$ be a chain of Lie groups. One can define the $n$-fold semidirect product multiplication on the set $G_1\x\dots \x G_n$ via
\begin{equation}
 (g_1, \ldots, g_{n}) \cdot (h_1, \ldots, h_{n}) = (g_1h_1, (\on{c}_{h_1\i} g_2) h_2,\ldots,(\on{c}_{(h_1\cdots h_{n-1})\i} g_n) h_n)
\end{equation}
with $ \on{c}_h g = hgh\i$ denoting conjugation. Then given the right-trivialised tangent vector $v(t) = (v_1(t),\ldots,v_n(t))$ of the curve $g(t)=(g_1(t),\ldots,g_n(t))$, the curve can be reconstructed via the ODE 
\begin{equation}
 \p_t g_k(t) = \Ad_{(g_1(t)\cdots g_{k-1}(t))\i} v_k(t)\enspace,
\end{equation}
if $k\geq 2$ and $\p_t g_1(t) = v(1)g(1)$.
Here $v_kg_k \in T_{g_k} G_k$ denotes the action of the group on its tangent space obtained by differentiating the left-multiplication. We shall denote this product by
\[ G_1 \ltimes \dots \ltimes G_n \]
to emphasize that each subproduct $G_k \ltimes \dots \ltimes G_n$ from the left is a normal subgroup of the whole product.
\end{lem}

\begin{proof}
This lemma can be proven in the same way as Lemma \ref{lem_sdp1}.
\end{proof}

These semidirect products defined in Lemmas \ref{lem_sdp1} and \ref{lem_sdp2} are equivalent as shown by the following lemma
\begin{lem} 
Let $G_1 \supseteq G_{2} \supseteq \ldots \supseteq G_n$ be a chain of Lie groups. The map
\begin{equation}
\Ph : \left\{ \begin{array}{ccc} 
G_1 \rtimes \dots\rtimes G_n & \to & G_n \ltimes \dots \ltimes G_1 \\
(g_1,\ldots,g_n) & \mapsto & (g_n, \ldots, \on{c}_{(g_{n+2-k}\dots g_{n})\i} g_{n+1-k},\ldots,\on{c}_{(g_2\dots g_n)\i} g_1)
\end{array} \right.
\end{equation}
is a group isomomorphism between the two semidirect products, and its derivative at the identity is given by
\begin{equation}
T_e\Ph : \left\{ \begin{array}{ccc} 
\g_1 \rtimes \dots\rtimes \g_n & \to & \g_n \ltimes \dots \ltimes \g_1 \\
(v_1,\ldots,v_n) & \mapsto & (v_n, \ldots, v_{n+1-k}, \ldots, v_1)
\end{array} \right. \enspace,
\end{equation}
with $\g_k$ being the Lie algebra of $G_k$ and $\g_1 \rtimes \dots\rtimes \g_n$ the Lie algebra of $G_1 \rtimes \dots\rtimes G_n$.
\end{lem}

\begin{proof}
Direct computation.
\end{proof}

The map $\Ph$ can be seen as one side of the following commutative triangle
\begin{equation} \xymatrix{ 
G_1 \rtimes \dots\rtimes G_n \ar[rr]^\Ph \ar[dr]_{T_1} && G_n \ltimes \dots\ltimes G_1 \ar[dl]^{T_2} \\
& G_1\times\dots\times G_n &
}
\end{equation}
The maps

\begin{align}
T_1(g_1,\ldots,g_n) &= (g_1\cdots g_n,g_2\cdots g_n,\ldots, g_{n-1}g_n, g_n) \\
T_2(g_n,\ldots,g_1) &= (g_n\cdots g_1,g_n\cdots g_2,\ldots,g_n g_{n-1}, g_n)
\end{align}

are group homomorphisms from the corresponding semidirect products into the direct product $G_1\times\dots \times G_n$. They can be regarded as trivialisation of the semidirect product in the special case that the factors form a chain of subgroups.

We will now assume that we are given $n$ kernels $\mathsf{k}_1,\ldots,\mathsf{k}_n$ with $H_{i} \subseteq H_{i+1}$, i.e. $\mathsf{k}_1$ represents the coarsest scale and $\mathsf{k}_n$ the finest one. Note that the inclusions are reversed as compared to Section \ref{sec_sdp}. The registration problem is then as follows.
\begin{defi}[LDDMM with the reversed order semidirect product]
\label{lddmm_sdp2}
$\null$\\Re\-gi\-ste\-ring  the image $I_0$ to the image $I_{\on{target}}$ is done by finding the minimising $n$-tuple $(v_1(t),\ldots,v_n(t))$ of
\[  \frac 12 \sum_{i=1}^n \int_0^1 \| v_i(t) \|^2_{H_i} \ud t + d^2( \ph(1).I_0,I_{\on{target}})\enspace,\]
where $\ph(t) = \ph_1(t)\o\dots\o\ph_n(t)$ and $\ph_k(t)$ is defined via
\begin{equation}
 \p_t \ph_k(t) = \left(\Ad_{(\ph_1(t)\o\dots\o\ph_{k-1}(t))\i}v_k(t)\right)\o\ph_k(t)\enspace.
\end{equation}
\end{defi}

To see that problems \ref{lddmm_sdp1} and \ref{lddmm_sdp2} are equivalent consider the commuting diagram
\begin{equation}
 \xymatrix{ 
G_1 \rtimes \dots\rtimes G_n \ar[rr]^\Ph \ar[dr]_{\ph=\ph_1\o\dots\o\ph_n} && G_n \ltimes \dots\ltimes G_1 \ar[dl]^{\ph=\ph_n\o\dots\o\ph_1} \\
& G_1 &
}
\end{equation}
and note that $T_e\Ph$ merely reverses the order of the vector fields $(v_1(t),\ldots,v_n(t))$. In particular the minimising vector fields $v_i(t)$ are the same (up to order) and the diffeomorphisms $\ph(1)$ coincide as well. The difference is in the diffeomorphisms $\ph_i(t)$ at each scale. We will see in Section \ref{scale_decomposition} that this version of the semidirect product is better suited to be generalised to a continuum of scales.

\section{Extension to a continuum of scales}
\label{cont_sec}

\subsection{The continuous mixture of kernels}
\label{cont_math}

In this section, we define the multi-scale approach for a continuum of scales. First, we introduce the necessary analytical framework and state some useful results.
\begin{defi}\label{def_admbundle}
Let $\Omega$ be a domain in $\R^{d}$. An \emph{admissible bundle} $(\mc{H},\lambda)$ is a couple consisting of a one-parameter family $\mc{H} := (H_s)_{s \in \R_+^*}$ (where $\R_+^* = ]0,+\infty[$) of admissible reproducing kernel Hilbert spaces of vector fields and a Borel measure $\lambda$ on $\R_+^*$, satisfying the following assumptions:
\begin{enumerate}
\item For any $s$, there exists a positive constant $M_s$ s.t. for every $v \in  H_s$,
\begin{equation}\label{rkhss}
\| v\|_{1,\infty} \leq M_s \| v\|_{H_s}\,.
\end{equation}
\item Denoting $\mathsf{k}_s$ the kernel of the space $H_s$,
the map $\R_+^* \times \R^d \times \R^d \ni (s,x,y) \mapsto \mathsf{k}_s(x,y) \in L(\R^d)$ is Borel measurable.
\item The map $s \mapsto M_s$ is Borel measurable with
\begin{equation} \label{FirstControl}
\int_{\R_+^*}\! M_s^2 \, \ud \lambda(s) < +\infty \,.
\end{equation}
\end{enumerate}
\end{defi}

\begin{remark} \label{RemOnContinuumSetting}
\begin{itemize}
\item Note that no inclusion is a priori required between the linear spaces $H_s$. However the typical example is given by the usual scale-space \emph{i.e.} $H_s$ defined by its Gaussian kernel $e^{-\frac{\|x-y\|^2}{2s^2}}$.
In this case, there exists an obvious chain of inclusions $H_{s} \subset H_{t}$ for $s>t>0$. This also explains our arbitrary choice of the parameter space which is $\R_+^*$. 
\item We have  $\int_{\R_+^*}\! \langle \alpha, \mathsf{k}_s(x,y)\beta \rangle \, \ud \lambda(s) \leq \abs{\alpha} \abs{\beta} \int_{\R_+^*} \! M_s^2 \, \ud \lambda(s)$. This follows from the Cauchy-Schwarz inequality $\langle \alpha, \mathsf{k}_s(x,y)\beta \rangle \leq \| \delta_x^\alpha \|_{H_s^*} \| \delta_y^\beta \|_{H_s^*} $ and the fact that $\| \delta_x^\alpha \|_{H_s^*} \leq |\alpha|M_s$. Recall that the notation $\delta_x^\alpha$ stands for the linear form defined by $(\delta_x^\alpha,f) = \langle f(x), \alpha \rangle$ for $x \in \Omega$ and $\alpha \in \R^d$.
\item The hypotheses in the definition may not be optimal to obtain the needed property, but this context is already large enough for applications.
\end{itemize}
\end{remark}

Mimicking Section \ref{FiniteNumberScales}, we consider the set of vector-valued functions defined on $\R_+^* \times \Omega$, namely denoting by $\mu$ the Lebesgue measure on $\Omega$,
\begin{equation} \label{MySpaceOfFunctions}
V := \left\{ u \in L^2(\R_+^* \times \Omega,\lambda \otimes \mu) \, \Biggl\vert \,\begin{array}{l}(1)\, \forall x \in \Omega \quad s \mapsto u(s,x) \text{ is measurable.}  \\  (2)\, s \mapsto \|u(s,\cdot)\|_{H_s} \text{ is measurable.} \\ (3) \, \int_{\R_+^*}\! \|u(s,\cdot)\|_{H_s}^2 \, \ud\lambda(s)< +\infty\,. \end{array} \right\} \,.
\end{equation}
It is rather standard to prove that $V$ is a Hilbert space for the norm defined by $ \|u\|^2_{V} := \int_{\R_+^*}\! \|u(s,.)\|_{H_s}^2 \, \ud\lambda(s)$.
Note that $V$ contains all the functions $y \mapsto  \mathsf{k}_s(x,y)\alpha$ for any $x \in \Omega$, $\alpha \in \R^d$ and $s \in \R^*_+$: Indeed, that function can be written as $K_s\delta_x^\alpha$ and its norm is then $\|\mathsf{k}_s(x,.)\alpha \|_{H_s}^2 = (\delta_x^\alpha,K_s \delta^\alpha_x) = \| \delta_x^\alpha \|_{H_s^*}^2$ which is finite by the second point of Remark \ref{RemOnContinuumSetting}.

Directly from the assumptions on the space $V$, we can define the set of vector-valued functions 
\begin{equation}
H := \left\{ x \mapsto \int_{\R_+^*}\! v(s,x) \, \ud \lambda(s) \,\biggl\vert \, v \in V \right\}\,.
\end{equation}

Remark that the integral $ \int_{\R_+^*}\! v(s,x) \, \ud \lambda(s)$ is finite using inequality \eqref{rkhss} and hypothesis \eqref{FirstControl} combined with the Cauchy-Schwarz inequality.

Then, the generalisation of Lemma \ref{ReductionLemma} (that generalisation can be found in  \cite{saitohbook,Schwartz64}) reads in our situation:

\begin{theo}\label{ReductionLemmaGeneralisation}
The space $H$ can be endowed with the following norm: For any $v \in H$,
\begin{equation} \label{NormDefinition}
\|v\|_{H}^2 = \inf_{u \in V} \int_{\R_+^*} \! \|u(s,\cdot)\|_{H_s}^2 \, \ud \lambda(s) \,,
\end{equation}
for $u$ satisfying the constraint $v = \int_{\R_+^*} \!u(s,\cdot)\, \ud \lambda(s)$.
This norm turns $H$ into a RKHS whose kernel is $\mathsf{k}(x,y) = \int_{\R_+^*}\! \mathsf{k}_s(x,y) \,\ud \lambda(s)$.
\end{theo}

In our case, the hypotheses on the bundle $(\mc{H},\lambda)$ imply that $H$ is an admissible RKHS.

\begin{proof}
As mentioned above, the linear map 
\begin{align*}
\rho:&\,V \to H \\ 
&\,v \mapsto \int_{\R_+^*} \!v_s\, \ud \lambda(s)
\end{align*}
is well-defined and continuous by the following inequality \eqref{RProp}, so is $\rho^*: H^* \mapsto V^*$. Using Cauchy-Schwarz\rq{}s inequality, we have
\begin{multline} \label{RProp}
\abs{\mathrm{ev}_{(x,\alpha)}(v)} = \left \lvert \int_{\R_+^*}\! \langle v_s(x),\alpha \rangle \,  \ud \lambda(s) \right \rvert  \leq  \abs{\alpha}\sqrt{\int_{\R_+^*}\! \|v(s,\cdot)\|_{H_s}^2 \,\ud \lambda(s)} \, \sqrt{\int_{\R_+^*}\! M_s^2 \,\ud \lambda(s)}\,
\end{multline}
so that, taking the infimum on the affine subspace of $v$ for a given $\rho(v)$, we have
\begin{equation}\label{rkhsProp}
\abs{\mathrm{ev}_{(x,\alpha)}(v)} \leq \abs{\alpha} \|\rho(v)\|_{H} \,  \sqrt{\int_{\R_+^*} \!M_s^2 \, \ud\lambda(s)}\,.
\end{equation}

Hence, the evaluation map $\mathrm{ev}_{(x,\alpha)} : v \in V \mapsto \int_{\R_+^*} \langle v(s,x),\alpha \rangle \,\ud\lambda(s) $ is continuous on $V$.
Therefore the map $\rho$ is continuous for the product topology on $(\R^d)^\Omega$ and its kernel is a closed subspace denoted by $S$. 

As a consequence of the orthogonal projection theorem and denoting by $\pi$ the orthogonal projection on $S^\perp$, we have:
For any $u \in V$, $\pi(u)$ is the unique element in $V$ such that
$\|\pi(u)\|^2_{V} = \inf_{v \in V} \{ \|v\|^2_{V} \,|\,\rho(v) = \rho(u) \}$. Therefore, Equation \eqref{NormDefinition} defines a norm on $H$ and $\rho_{S^\perp}:S^\perp \mapsto H$ is an isometry. In particular, $H$ is a Hilbert space.

Note that inequality \eqref{rkhsProp} shows that $H$ is a RKHS. A direct consequence of this fact is that $H^*\hookrightarrow H_s^*$ a.e. on $\R^*_+$: Indeed, we have
$\delta_x^\alpha \in H_s^* \cap H^*$ and since the span (denoted by $D$) of all such elements is dense in $H^*$ (because $H$ is a RKHS), for any element $p \in H^*$ there exists a Cauchy sequence $p_k \in D$ converging to $p$ in $H^*$ so that $\mu$-a.e. $p_k$ is a Cauchy sequence in $H_s^*$. Denoting by $(\cdot,\cdot)$ the duality pairing,, we have that $(\rho^*(p),v)= \lim_{k \rightarrow +\infty} (\rho^*(p_k),v)$. Then, we deduce for any $v \in V$ 
\begin{equation}
(\rho^*(p),v)=  \lim_{k \rightarrow +\infty} \int_{\R_+^*} \left(p_k,v(s,\cdot)\right) \, \ud\lambda(s) = \int_{\R_+^*} \left(p,v(s,\cdot)\right) \, \ud\lambda(s)\,
\end{equation}
by application of the dominated convergence theorem and the Cauchy-Schwarz inequality.

We now introduce the Lagrangian
\begin{equation}
\ell_u(p,v)=\frac{1}{2}\|v\|_{V}^2 + \left( p, u - \rho(v)\right) \,.
\end{equation}
where $(u,p) \in H \times H^*$ and $v \in V$. It can be checked easily that $(v^\star=\rho_{S^\perp}^{-1}(u),p^\star)$ with $u=Kp^\star$ ($K$ being the Riesz isomorphism between $H^*$ and $H$) is a stationary point of $\ell_u$. So that, we obtain $\rho(v^\star) = u$ and $\lambda$ a.e., $ K_s p^\star  = v^\star_s$.

The Riesz isomorphism is therefore given by
\begin{equation}
K:p \in H^* \mapsto \int_{\R_+^*}\! K_s p \, \ud \lambda(s) \in H \,,
\end{equation}
and the kernel function is given by
\begin{equation}
\mathsf{k}(x,y) =( \delta_x, K \delta_y ) = \int_{\R_+^*} \! \mathsf{k}_s(x,y) \, \ud \lambda(s) \,.
\end{equation}

\end{proof}
\begin{remark}
The hypotheses on the bundle $(\mc{H},\lambda)$ imply that $H$ is an admissible RKHS since we can apply a theorem of differentiation under the integral:
By the hypothesis on the RKHS $H_s$, $\lambda$ a.e. on $\R_+^*$ the map $x \mapsto v(s,x)$ is differentiable at any point $x_0 \in \Omega$ and $\|v(s,\cdot) \|_{1,\infty} \leq M_s\|v(s,\cdot)\|_{H_s}$.
By application of Cauchy-Schwarz\rq{}s inequality, the right-hand side is integrable so that $\rho(v)$ is also differentiable and $C^1$ and
\begin{equation}
\|\rho(v)\|^2_{1,\infty} \leq   \int_{\R_+^*}\! M_s^2 \, \ud \lambda(s) \int_{\R_+^*}\! \| v_s \|^2_{H_s^*}\,  \ud \lambda(s) \,.
\end{equation}
\end{remark}

\subsection{Scale decomposition}
\label{scale_decomposition}

In this section, we will generalise the ideas of Section \ref{sec_sdp} from a finite sum of kernels to a continuum of
scales. We will make some more assumptions to the general setting introduced in Section \ref{cont_math}.
\begin{assump}
We assume that the measure $\la(s)$ is the Lebesgue measure on the finite interval $[0,1]$ and that the family $H_s$ of RKHS is ordered by inclusion
\[ H_s \subseteq H_t \text{ for } s \leq t\,. \]
\end{assump}
This assumption might be relaxed a little bit: As long as $\la(s)$ is absolutely continuous with respect to the Lebesgue measure, i.e. it can be represented via a density $\la(s) = f(s) \ud s$, the same construction can be carried out. The ordering of the inclusions corresponds to that in Section \ref{sec_reverse}.

As in the discrete setting we can formulate the two image matching problems. The first is a direct generalisation of problem \ref{lddmm_sumofker} to a continuum of scales.
\begin{defi}[LDDMM with an integral over kernels]
\label{lddmm_intofker}
Registering the $I_0$ to the image $I_{\on{target}}$ is done by finding the minimiser $v(t)$ of
\[  \frac 12 \int_0^1 \| v(t) \|^2_{H} \ud t + d^2( \ph(1).I_0,I_{\on{target}})\enspace,\]
where $\ph(t)$ is the flow of $v(t)$ and  $\mathsf{k}=\int_0^1\mathsf{k}_s \ud s$ is the integral over the scales and $H$ the corresponding RKHS.
\end{defi}
The other problem associates to each scale a separate vector field. It was proposed in \cite{SommerKernelBundle1, SommerKernelBundle2}, where it was called registration with a kernel bundle. The term kernel bundle refers to the one-parameter family $\mc H = (H_s)_{s\in[0,1]}$ of RKHS.
\begin{defi}[LDDMM with a kernel bundle]
\label{lddmm_intsimmultscale}
Registering the image $I_0$ to the image $I_{\on{target}}$ is done by finding the one-parameter family $v_s(t)$ of vector fields, which minimises
\[  \frac 12 \int_0^1 \int_0^1 \| v_s(t) \|^2_{H_s}  \ud s \ud t + d^2( \ph(1).I_0,I_{\on{target}})\enspace,\]
where $\ph(t)$ is the flow of the vector field $v(t) = \int_0^1 v_s(t) \ud s$.
\end{defi}

These two problems are equivalent, as will be shown in Theorem \ref{cont_equiv}. As a next step we want to obtain a geometric reformulation of the registration problem similar to the problem statements \ref{lddmm_sdp1} or \ref{lddmm_sdp2}. The goal of this reformulation is to decompose the minimising flow of diffeomorphisms $\ph(t)$, such that the effect of each scale becomes visible. In order to do this decomposition we define for fixed $s$:
\begin{equation}
\label{def_psi}
 \ps_s(t) \text{ is the flow of } \int_0^s v_r(t) \ud r \text{ in $t$}\enspace.
\end{equation}
The following theorem allows us to interchange time and scale in the flow $\ps_s(t)$.
\begin{theo}
\label{switch_st}
For each fixed $t$, the one-parameter family $s \mapsto \ps_s(t)$ is the flow in $s$ of the vector field
\[ \Ad_{\ps_s(t)} \int_0^t \Ad_{\ps_s(r)\i} v_s(r) \ud r\enspace.\]
\end{theo}

To prove this theorem we will use the following lemma.
\begin{lem}
\label{switch_st_lemma}
Let $u(s,t,x)$ and $v(s,t,x)$ be two-parameter families of vector fields which are $C^2$ in the $(s,t)$-variables and $C^1$ in $x$. If they satisfy
\begin{equation}
 \p_s u(s,t,x) - \p_t v(s,t,x) = [u(s,t), v(s,t)](x), 
\end{equation}
where $[u,v] = Dv.u - Du.v$ is the Lie algebra bracket (minus the usual Jacobi bracket) for vector fields and $Du$ is the Jacobi matrix of $u:\Om \to \R^d$, and if $v(s,0)\equiv 0$ for all $s$, then the flow of $u(s,.)$ for fixed $s$ coincides with the flow of $v(.,t)$ for fixed $t$.
\end{lem}

\begin{proof}
Denote by $a_s(t)$ the flow of $u(s,.)$ in $t$. Then
\begin{align*}
\p_t\p_s a_s(t) &= \p_s\p_t a_s(t) = \p_s(u(s,t)\o a_s(t)) \\
&= \p_su(s,t)\o a_s(t) + Du(s,t,a_s(t)).\p_sa_s(t) \\
&= \p_t v(s,t)\o a_s(t) + [u(s,t),v(s,t)]\o a_s(t) + Du(s,t,a_s(t)).\p_sa_s(t) \\
&= \p_t \left(v(s,t)\o a_s(t) \right) + Du(s,t,a_s(t)).\left(\p_sa_s(t) - v(s,t)\o a_s(t) \right)\,.
\end{align*}
This implies that $b_s(t) := \p_sa_s(t) - v(s,t)\o a_s(t)$ is the solution of the ODE
\begin{equation} 
\label{helper_ode}
\p_t b_s(t) = Du(s,t,a_s(t)).b_s(t)\enspace.
\end{equation}
Since for $t=0$ we have $b_s(0) = \p_s a_s(0) - v(s,0) \o a_s(0) = 0$, it follows that $b_s(t) \equiv 0$ is the unique solution of \eqref{helper_ode}. This means that
\begin{equation} \p_sa_s(t) = v(s,t)\o a_s(t)\,,
\end{equation}
i.e. the flows of $u(s,.)$ in $t$ and of $v(.,t)$ in $s$ coincide.
\end{proof}

\begin{proof}[Proof of Theorem \ref{switch_st}]
We apply Lemma \ref{switch_st_lemma} to the vector fields 
\[ \int_0^s v_r(t)\ud r \text{ and } \Ad_{\ps_s(t)} \int_0^t \Ad_{\ps_s(r)\i} v_s(r) \ud r\enspace. \]
We can differentiate $\Ad$ using the following rule
\begin{equation}
 \p_t \Ad_{g(t)} u = [\p_t g(t)g(t)\i, \Ad_{g(t)} u]\enspace.
\end{equation}

This can be seen by writing
\begin{align*}
\p_t \Ad_{g(t)} u|_{t=t_0} &= \p_t \Ad_{g(t)g(t_0)\i} \Ad_{g(t_0)} u|_{t=t_0} \\
&= \ad_{\p_t g(t)g(t_0)\i|_{t=t_0}} \Ad_{g(t_0)} u \\
&= [ \p_t g(t)g(t_0)\i|_{t=t_0}, \Ad_{g(t_0)} u]\enspace.
\end{align*}

Using this we can verify the compatibility condition
\begin{align*}
\p_s &\left( \int_0^s v_r(t)\ud r \right) - \p_t \left( \Ad_{\ps_s(t)} \int_0^t \Ad_{\ps_s(r)\i} v_s(r) \ud r \right) = \\
&= v_s(t) - \left[\p_t \ps_s(t) \ps_s(t)\i, \Ad_{\ps_s(t)} \int_0^t \Ad_{\ps_s(r)\i} v_s(r) \ud r\right] - v_s(t) \\
&= \left[\int_0^s v_s(t) \ud r, \Ad_{\ps_s(t)} \int_0^t \Ad_{\ps_s(r)\i} v_s(r) \ud r \right] \enspace .
\end{align*}
The condition $v(s,0)\equiv 0$ is trivially satisfied. This concludes the proof.
\end{proof}

Theorem \ref{switch_st} gives us a way to decompose the matching diffeomorphism $\ph(1)$ into separate scales. As we follow the flow $\ps_s(1)$, we add more and more scales, starting from the identity at $s=0$, when no scales are taken into account and finishing with $\ph(1)$ at $s=1$, which includes all scales. In this sense $\ps_s(1)\i\o\ps_t(1)$ contains the scale information for the scales in the interval $[s,t]$.

\subsection{Restriction to a finite number of scales}
\label{ssec:restrictFinitNbScales}

It is of interest to understand the relationship between a continuum of scales and the case, where we have only a finite number. We will see, that it is possible to see the discrete number of scales as a special case of the continuum of kernels.

Let us start with a family $\mathsf{k}_s$ of kernels with $s\in [0,1]$, where the scales are ordered from the coarsest to the finest, i.e. $H_{s} \leq H_{t}$ for $s \leq t$ as before. Divide the interval $[0,1]$ into $n$ parts $0=s_0 < \ldots < s_n = 1$ and denote the intervals $I_k = [s_{k-1},s_k]$. Let us consider the space
\begin{equation}
 V = \left\{ v\in \mc H: \int_0^1 \| v_s \|^2_{H_s} \ud s < \infty \right\}\enspace,
\end{equation}
which was defined in \eqref{MySpaceOfFunctions} and $\mc H$ in Definition \ref{def_admbundle} to be a one-parameter family of RKHS $H_s$. To each interval $I_k$ corresponds a kernel $\int_{I_k} \mathsf{k}_s \ud s$ and a RKHS $H_k$. The discrete sampling map
\begin{equation}
 \Ps: \left\{ \begin{array}{ccc}
V & \to & H_{1} \times \dots \times H_{n} \\
v & \mapsto & (\int_{I_1} v_s \ud s, \ldots, \int_{I_n} v_s \ud s)
\end{array} \right.\
\end{equation}
discretises $v_s$ into $n$ scales. Formally we can introduce a Lie bracket on the space $V$ by defining
\begin{equation}
\label{cont_bracket}
 [u, v]_s = \left[u_s, \int_0^s v_r \ud r\right] + \left[\int_0^s u_r \ud r, v_s\right]\enspace.
\end{equation}
Using this bracket the sampling map $\Ps$ is a Lie algebra homomorphism as shown in the next theorem.

\begin{theo}
The sampling map $\Ps$ is a Lie algebra homomorphism from the Lie algebra $V$ with the bracket defined in \eqref{cont_bracket} into the $n$-fold semidirect product with the bracket
\begin{equation}
 [(u_1,\ldots,u_n), (v_1,\ldots,v_n)] = ([u_1,v_1], \ldots, [u_k,\sum_{i=1}^{k-1} v_i] + [\sum_{i=1}^{k-1} u_i, v_k] + [u_k,v_k], \ldots )\,.
\end{equation}
\end{theo}

\begin{proof}
Using the definitions we first compute
\begin{align*}
[\Ps(u), \Ps(v)]_k &= \left[\int_{s_{k-1}}^{s_k} u_s \ud s, \int_0^{s_{k-1}} v_s \ud s\right] + \left[\int_{0}^{s_{k-1}} u_s \ud s, \int_{s_{k-1}}^{s_{k}} v_s \ud s\right]  \\
&\phantom{=\quad} + \left[\int_{s_{k-1}}^{s_k} u_s \ud s, \int_{s_{k-1}}^{s_{k}} v_s \ud s\right] \\
&= \left[\int_{0}^{s_k} u_s \ud s, \int_0^{s_{k}} v_s \ud s\right] - \left[\int_{0}^{s_{k-1}} u_s \ud s, \int_0^{s_{k-1}} v_s \ud s\right] \enspace ,
\end{align*}
and then write the other side
\begin{align*}
\Ps([u,v])_k &= \int_{s_{k-1}}^{s_k} \left[u_s, \int_0^s v_r \ud r\right] + \left[\int_0^s u_r \ud r, v_s\right] \ud s
\end{align*}
Below we interchange the order of integration in the first summand and merely switch $s$ and $r$ in the second summand to obtain
\begin{align*}
\int_{0}^{s_k} \left[u_s, \int_0^s v_r \ud r\right] + \left[\int_0^s u_r \ud r, v_s\right] \ud s &= \int_0^{s_k} \int_0^{s_k} 1_{r \leq s} [u_s, v_r] \ud r \ud s \\&+ \int_0^{s_k} \int_0^{s_k} 1_{s\leq r} [u_s, v_r] \ud s \ud r \\
&= \int_0^{s_k} \int_0^{s_k} [u_s, v_r] \ud s \ud r \\
&= \left[ \int_0^{s_k} u_s \ud s, \int_0^{s_k} v_s \ud s \right] \enspace.
\end{align*}
Decomposing the integral into
\[ \Ps([u,v])_k = \int_0^{s_k} \dots \ud s - \int_0^{s_{k-1}} \dots \ud s \]
finishes the proof.
\end{proof}

Now we can show that all matching problems that we defined in the continuous case are equivalent.

\begin{theo}
\label{cont_equiv}
The matching problems \ref{lddmm_intofker} and \ref{lddmm_intsimmultscale} are equivalent and using the sampling map $\Ps$ they are also equivalent to the discrete problem \ref{lddmm_sdp2}.
\end{theo}

\begin{proof}
The first equivalence follows from Theorem \ref{ReductionLemmaGeneralisation}. For the second equivalence note that problem \ref{lddmm_sdp2} is equivalent to \ref{lddmm_sumofker} and use Lemma \ref{ReductionLemma}.
\end{proof}

The diffeomorphisms $\ph_k$ at each scale, that were defined in \ref{lddmm_sdp2} are also contained in the continuous setting. If $u_k = \int_{I_k} v_s \ud s$ is the $k$-th component of the sampling map, then $\ph_1(t)\o\dots\o\ph_k(t)$ is the flow of the vector field $u_1(t)+\dots+u_k(t)$ and we have
\begin{equation}
 u_1(t)+\dots+u_k(t)  = \int_0^{s_k} v_s(t) \ud s\enspace.
\end{equation}
Hence we obtain the identity $\ph_1(t)\o\dots\o\ph_k(t) = \ps_{s_k}(t)$, where $\ps_{s_k}$ was defined in \eqref{def_psi}. In particular we retrieve
\begin{equation}\label{eq:linkPsiPhi}
 \ph_k(t) = \ps_{s_{k-1}}(t)\i\o\ps_{s_k}(t)
\end{equation}
the scale decomposition of the discrete case as a continuous scale decomposition evaluated at specific points.

\section{Conclusion and outlook}
\label{sec:conclusionAndOutlook}

In this paper, we have extended the mixture of kernels presented in \cite{Risser_MK} to  the continuous setting and we have given a variational interpretation of the matching problem. In particular, we have shown that the approaches presented in \cite{SommerKernelBundle1,SommerKernelBundle2} and the mixture of kernels of \cite{Risser_MK} are equivalent.
Motivated by the mathematical development of the multi-scale approach to group of diffeomorphisms, we have extended the semidirect product result of \cite{Bruveris_MK} to more than two discrete scales and also to a continuum of scales. 
In simulations on both synthetic and biological images, we have shown that the extracted diffeomorphisms at each separate scale reveal very interesting information on the structure of the final diffeomorphism. 

Further work will deal with the statistical use of this decomposition. In particular, we will work on the multi-scale description of the variability of organs. This will extend the work of  \cite{VialardMIUA11}, where the authors defined average shapes in the Riemannian framework of the \emph{Large Deformation Diffeomorphic Metric Mapping} also used in the present paper. Defining the average shape of an organ and its variability in a group of subjects is of particular interest in medical image analysis since it allows the automatic detection of abnormalities. The present work opens new perspectives for the multi-scale detection of abnormalities. 

Another direction will be the development of statistics on the initial momentum with respect to the mixture of kernels.   Based on promising results from previous work \cite{RisserMICCAI,Risser_MK}, we plan to expand the use of the multi-scale information a statistical context.
To this aim, it would be very interesting to work on the definition of kernels with sparsely distributed scales 
to improve the statistical power of the scale related deformations or the initial momenta. 
This is in analogy with results in \cite{NIPS2009_0879}, although the conclusions of \cite{Risser_MK} tend to favour non-sparse descriptions of the scales from a purely image matching point of view.
In this direction, more theoretical approaches to learn the scales and weights  involved in the mixture of kernels will be developed in the future. 

\subsection*{Acknowledgements}
We thank Alain Trouv\'e for his insightful remarks on the paper and the reviewer, whose extremely valuable comments helped us to greatly improve the content and the presentation of this work.

\appendix

$\null$\\

\section{Multiscale registration algorithm}\label{app:MultiRegAlgo}

In this appendix, we describe how we register a source/template image $I$ to a target image $J$ when using a kernel $\mathsf{k}$ constructed as the weighted sum of $k$ Gaussian kernels $\mathsf{k}_i(x,y) = a_i  e^{-\|x-y\|^2/\sigma_i^2}Id$. More specifically, we present how we obtain the time dependent velocity fields $u_i(t)$, $i \in \{1, \cdots ,k\}$ on which the scale related diffeomorphisms are built, as show in Section~\ref{ssec:restrictFinitNbScales}.
This algorithm was presented in \cite{Risser_MK}.
Its implementation is also freely available on sourceforge\footnote{http://sourceforge.net/projects/utilzreg/}.
We  introduce $w_i(t)$ as the velocity field updates of $u_i(t)$. The time $t \in [0,1]$ is linearly discretised into $\Theta$ time steps $t_{\theta}$. The registration algorithm is as follows:\\

\begin{algorithmic}
\STATE \COMMENT{{\it Initialisation}}
\STATE The time-dependent velocity fields $u_i$ are initialised to 0
\REPEAT
  \STATE \COMMENT{{\it Compute the mappings and project $\{I,J\}$}}
  \STATE Compute $u (t_{\theta}) = \sum_{i=1}^k u_i (t_{\theta})$ for all $t_{\theta}$
  \STATE Compute the mappings $\psi$ and $\psi^{-1}$ for all $t_{\theta}$ using Eq.~\eqref{eq:integrationVF}
  \STATE Compute $\{I_{t_{\theta}},J_{t_{\theta}}\}$, the projections of $\{I,J\}$ for all $t_{\theta}$ using $\psi$ and $\psi^{-1}$. 
  \FOR{$\theta = 1 \to \Theta $}
    \STATE \COMMENT{{\it Compute the gradient of $u(t_{\theta})$}}
    \STATE $w(t_{\theta}) \leftarrow \left(   | D \psi_{t_{\theta}, 1} |  \nabla I_{t_{\theta}}  (I_{t_{\theta}} - J_{t_{\theta}})     \right)$
    \FOR{$i = 1 \to k$}
      \STATE \COMMENT{{\it Scale dependent smoothing of $w(t_{\theta})$}}
      \STATE $w_i(t_{\theta})  \leftarrow \mathsf{k}_i  \star w(t_{\theta})$
      \STATE \COMMENT{{\it  Update  $u_i$}}
      \STATE $u_i (t_{\theta}) \leftarrow u_i (t_{\theta}) - \epsilon w_i (t_{\theta})$
    \ENDFOR
  \ENDFOR
\UNTIL{Convergence}
\end{algorithmic}

Note that $\epsilon$ is a scalar chosen so that $\max (\epsilon w_i (t_{\theta}))$ is of the order of a voxel at the beginning of the gradient descent.
The integration of the velocity field (Eq.~\eqref{eq:integrationVF}) can be made as follows: to compute $\psi(t_{\theta})$, we first define $\psi(t_1=0)$ as an identity deformation. We then incrementally estimate $\psi(t_{\tau+1})$ using an Euler or a Leap-frog scheme with $\psi(t_{\tau})$, $u(t_{\tau-1})$ and eventually $u(t_{\tau})$, until $t_{\tau}=t_{\theta}$.

$\null$\\

\section{Tuning of the weights}\label{app:WeightsTuning}

Our multi-kernel registration technique depends on  a set of parameters $a_i$, $i \in [1,k]$, each of them controlling the weight of the deformations at a scale $i$. 
As shown in Appendix~\ref{app:MultiRegAlgo}, the gradients of the optimisation algorithm to estimate the velocity fields $v_i$ not only depend on the images $I$ and $J$ but also on the $\mathsf{k}_i$, which are parametrised by the $\sigma_i$ and the $a_i$. Once the characteristic scales $\sigma_i$ are chosen to compare $I$ and $J$, the tuning of the $a_i$ depends on (1) the representation and spatial organisation of the structures in $I$ and $J$ and (2) a knowledge about the expected maximum amplitude of the structures displacement at each scale. We therefore consider $a_i = a_i' / g(\sigma_i,I,J)$, where $g(\sigma_i,I,J)$ is related to (1) and the apparent weight $a_i'$ is related to (2).\\

As described in \cite{Risser_MK}, the apparent weights $a_i'$ are introduced to provide an intuitive control on the tuning of the $a_i$. The maximum amplitude of the deformations should be similar at all scales if all the $a_i'$ are equal. Variations of $a_i'$ should be linearly related with the maximum amplitude of the deformations captured at scale $i$. To do so, we empirically tune the $g(\sigma_i,I,J)$ so that, if all the $a_i'$ equal $1$, then the maximum value of the velocity update $w_i$ (see appendix~\ref{app:MultiRegAlgo}) is the same for all $i$ at the first iteration of the gradient descent algorithm. We can then show that the $g(\sigma_i,I,J)$ can be quickly computed using:
\begin{equation}
g(\sigma_i,I,J) = \max \left( \mathsf{k}_i \star \nabla I  (I - J)  \right) \, ,
\end{equation}
where the smoothing kernel $\mathsf{k}_i$ is not weighted. This strategy is applied once, prior to the gradient descent.\\

This strategy can be slighly modified when a template image $I$ is compared with $M$ images $J_m$ for statistical purposes. In this case, the smoothing kernel $\mathsf{k}$ must be the same for the comparison of all image pairs $\{I,J_m\}$ and the average $g(\sigma_i,I,J_m)$ can be chosen to tune $\mathsf{k}_i$. The strategy described in this paper to distinguish the deformations captured at different scales therefore allows to perform multiscale comparisons with more or less emphasis on each scale $i$ according to the apparent weight $a_i'$. As discussed in Section \ref{sec:conclusionAndOutlook}, this is the main perspective of this work.

\addcontentsline{toc}{section}{References}
\small{
\newcommand{\etalchar}[1]{$^{#1}$}

\bibliographystyle{alpha}
}
\end{document}